\documentclass[reqno]{amsart}
\usepackage{amsmath, amsthm, amssymb, amstext}

\usepackage{hyperref,xcolor}
\hypersetup{
 pdfborder={0 0 0},
 colorlinks,
}
\usepackage{enumitem}
\newtheorem{theorem}{Theorem}
\newtheorem{remark}[theorem]{Remark}
\newtheorem{lemma}[theorem]{Lemma}

\newtheorem{definition}[theorem]{Definition}

\DeclareMathOperator*{\divergenz}{div}              %
\DeclareMathOperator*{\Ss}{S}

\newcommand{\N}{\mathbb{N}}

\newcommand{\R}{\mathbb{R}}

\newcommand{\RN}{\mathbb{R}^N}

\newcommand{\Lp}[1]{L^{#1}(\Omega)}

\newcommand{\Wp}[1]{W^{1,#1}(\Omega)}

\newcommand{\Wpzero}[1]{W^{1,#1}_0(\Omega)}

\newcommand{\ph}{\varphi}
\newcommand{\Om}{\Omega}
\newcommand{\rand}{\partial\Omega}

\newcommand{\into}{\int_{\Omega}}

\newcommand{\weak}{\rightharpoonup}

\newcommand{\close}{\overline{\Omega}}

\renewcommand{\l}{\left}
\renewcommand{\r}{\right}
\newcommand{\Hi}{\mathcal H_i}
\numberwithin{theorem}{section}
\numberwithin{equation}{section}

\title[Existence and uniqueness of elliptic systems]{Existence and uniqueness of elliptic systems with double phase operators and convection terms}

\author[G.\,Marino]{Greta Marino}
\address[G.\,Marino]{Technische Universit\"{a}t Chemnitz, Fakult\"{a}t f\"{u}r Mathematik, Reichenhainer Stra\ss e 41, 09126 Chemnitz, Germany}
\email{greta.marino@mathematik.tu-chemnitz.de}

\author[P.\,Winkert]{Patrick Winkert}
\address[P.\,Winkert]{Technische Universit\"{a}t Berlin, Institut f\"{u}r Mathematik, Stra\ss e des 17.\,Juni 136, 10623 Berlin, Germany}
\email{winkert@math.tu-berlin.de}



\begin{document}

\begin{abstract}
        In this paper we study quasilinear elliptic systems driven by so-called double phase operators and nonlinear right-hand sides depending on the gradients of the solutions. Based on the surjectivity result for pseudomonotone operators we prove the existence of at least one weak solution of such systems. Furthermore, under some additional conditions on the data, the uniqueness of weak solutions is shown.   
\end{abstract}

\maketitle

\section{Introduction}

In this paper, given a bounded domain $\Om \subset \RN$, $N\geq 2$, with a Lipschitz boundary $\rand$, we are concerned with the existence and uniqueness of solutions to the following elliptic system
	\begin{align}
	\label{syst}
	\begin{aligned}
	-\divergenz \l(|\nabla u|^{p_1-2} \nabla u+ \mu_1(x) |\nabla u|^{q_1-2} \nabla u\r)&= f_1(x, u, v, \nabla u, \nabla v) \quad && \text{in } \Om, \\
	-\divergenz \l(|\nabla v|^{p_2-2} \nabla v+ \mu_2(x) |\nabla v|^{q_2-2} \nabla v\r)&= f_2(x, u, v, \nabla u, \nabla v) \quad && \text{in } \Om, \\
	u&= v= 0 && \text{on } \rand,
	\end{aligned}
	\end{align}
where $1< p_i< q_i< N$, $\mu_i\colon \close \to [0, \infty)$ are Lipschitz continuous and $f_i\colon \Omega \times \R \times \R \times \RN \times \RN \to \R$ are Carath\'eodory functions, $i= 1, 2$, that satisfy suitable structure conditions, see hypotheses (H) in Section \ref{section_3}.

Here, the operator is the so-called double-phase operator, that is
\begin{align*}
        -\divergenz \l(|\nabla u|^{p-2} \nabla u+ \mu(x) |\nabla u|^{q-2} \nabla u\r)\quad \text{for }u\in \Wp{\mathcal{H}},
\end{align*}
where $1<p<q<N$ and with a suitable Sobolev Musielak-Orlicz space $\Wp{\mathcal{H}}$, see its definition in Section \ref{section_2}. Such an operator is the extension of the so-called weighted $(q,p)$-Laplacian when $\inf_{\close} \mu>0$ and of the $p$-Laplace differential operator when $\mu\equiv 0$. 

The novelty of this work is an existence and uniqueness result for problems of the form \eqref{syst} by using the surjectivity result for pseudomonotone operators, see Definition \ref{SplusPM} and Theorem \ref{theorem_pseudomonotone}. To the best of our knowledge, this is the first work dealing with a double phase operator and a convection term (that is, the right-hand side depends on the gradient of the solution) in the context of elliptic systems. 

 Zhikov \cite{Zhikov-1986} was the first who studied so-called double phase operators in order to describe models of strongly anisotropic materials by studying the functional
\begin{align}\label{integral_minimizer}
   u \mapsto \int \left(|\nabla  u|^p+\mu(x)|\nabla  u|^q\right)\,dx,
\end{align}
where $1<p<q<N$, see also Zhikov \cite{Zhikov-1995}, \cite{Zhikov-1997} and the monograph of Zhikov-Kozlov-Oleinik \cite{Zhikov-Kozlov-Oleinik-1994}. Functionals of the expression \eqref{integral_minimizer} have been studied by several
 authors with respect to regularity results and nonstandard growth, see for example,
 Baroni-Colombo-Mingione \cite{Baroni-Colombo-Mingione-2015}, \cite{Baroni-Colombo-Mingione-2016}, \cite{Baroni-Colombo-Mingione-2018},
 Baroni-Kuusi-Mingione \cite{Baroni-Kuusi-Mingione-2015},
 Cupini-Marcellini-Mascolo \cite{Cupini-Marcellini-Mascolo-2015},
 Co\-lom\-bo-Mingione \cite{Colombo-Mingione-2015a}, \cite{Colombo-Mingione-2015b},
 Marcellini \cite{Marcellini-1989},
 \cite{Marcellini-1991} and the references therein.

 The motivation of this work was on the one hand the work of Gasi\'nski-Winkert \cite{Gasinski-Winkert-2020b} who proved existence and uniqueness for the problem
 \begin{equation}\label{problem}
    \begin{aligned}
	-\divergenz\left(|\nabla u|^{p-2}\nabla u+\mu(x) |\nabla u|^{q-2}\nabla u\right) & =f(x,u,\nabla u)\quad && \text{in } \Omega,\\
	u & = 0 &&\text{on } \partial \Omega,
    \end{aligned}
\end{equation}
following the paper of Averna-Motreanu-Tornatore \cite{Averna-Motreanu-Tornatore-2016}. On the other side, we were also motivated by the paper of Motreanu-Vetro-Vetro \cite{Motreanu-Vetro-Vetro-2016} who treated elliptic systems for $(p_i,q_i)$-Laplace operators of the form
\begin{equation}\label{problem2}
        \begin{aligned}
                -\Delta_{p_1}u_1-\mu_1 \Delta_{q_1} u_1 & =f_1(x,u_1,u_2,\nabla u_1, \nabla u_2)\quad && \text{in } \Omega,\\
                -\Delta_{p_2}u_2-\mu_2 \Delta_{q_2} u_2 & =f_2(x,u_1,u_2,\nabla u_1, \nabla u_2)\quad && \text{in } \Omega,\\
                u_1=u_2 & = 0 &&\text{on } \partial \Omega.
    \end{aligned}
\end{equation}
The idea in the current paper is to combine both problems \eqref{problem} and \eqref{problem2} which gives our model problem \eqref{syst}. Such new class of problems brings lots of difficulties to be overcome like the Orlicz space in order to deal with the double phase operator, the gradient dependence of the right-hand side which implies that we cannot use variational tools and the fact that we treat this for elliptic systems. Our results extend those in Gasi\'nski-Winkert \cite{Gasinski-Winkert-2020b} and Motreanu-Vetro-Vetro \cite{Motreanu-Vetro-Vetro-2016}.

In the case of single-valued equations like \eqref{problem} without convection we refer to the works of Colasuonno-Squassina \cite{Colasuonno-Squassina-2016}, Gasi\'nski-Papageorgiou \cite{Gasinski-Papageorgiou-2019}, Gasi\'nski-Winkert \cite{Gasinski-Winkert-2020a},  Liu-Dai \cite{Liu-Dai-2018}, Perera-Squassina \cite{Perera-Squassina-2019} concerning existence and multiplicity results.

Elliptic systems with the shape as in \eqref{problem2} have been considered by a number of authors. Existence results can be found, for examples, in Boccardo-de Figueiredo \cite{Boccardo-de-Figueiredo-2002}, Carl-Motreanu \cite{Carl-Motreanu-2017}, \cite{Carl-Motreanu-2015}, Dr\'{a}bek-Stavrakakis-Zographopoulos \cite{Drabek-Stavrakakis-Zographopoulos-2003}, Motreanu-Vetro-Vetro \cite{Motreanu-Vetro-Vetro-2018} and the references therein.

 Works which are closely related to our paper dealing with certain types
 of double phase problems, convection terms or elliptic systems can be found in Bahrouni-R\u{a}dulescu-Repov\v{s} \cite{Bahrouni-Radulescu-Repovs-2019}, Bahrouni-R\u{a}dulescu-Winkert \cite{Bahrouni-Radulescu-Winkert-2019}, Cencelj-R\u{a}dulescu-Repov\v{s} \cite{Cencelj-Radulescu-Repovs-2018}, Marano-Marino-Moussaoui \cite{Marano-Marino-Moussaoui-2019}, Marano-Winkert \cite{Marano-Winkert-2019}, Marino-Winkert \cite{Marino-Winkert-2020}, Motreanu-Winkert \cite{Motreanu-Winkert-2019}, Papageorgiou-R\u{a}dulescu-Repov\v{s} \cite{Papageorgiou-Radulescu-Repovs-2019a}, \cite{Papageorgiou-Radulescu-Repovs-2019b}, \cite{Papageorgiou-Radulescu-Repovs-2020}, R\u{a}dulescu \cite{Radulescu-2019}, Zhang-R\u{a}dulescu \cite{Zhang-Radulescu-2018}, Zheng-Gasi\'nski-Winkert-Bai \cite{Zheng-Gasinski-Winkert-Bai-2020} and the references therein.
 
 The paper is organized as follows. In Section \ref{section_2}  we recall the definition of the Musielak-Orlicz spaces $\Lp{\mathcal{H}}$ and its corresponding Sobolev spaces $\Wp{\mathcal{H}}$ and we recall the surjectivity result for pseudomonotone operators. In Section \ref{section_3} we present the full assumptions on the data of problem \eqref{syst},
 give the definition of the weak solution and state and prove our main existence result, see Theorem \ref{main}. In the last part, namely Section \ref{section_4}, we state some conditions on $f_i$, $i=1,2$, in order to prove the uniqueness of weak solutions of \eqref{syst}, see Theorem \ref{theorem_uniqueness}.

\section{Preliminaries}\label{section_2}

For every $1 \le r< \infty$ we consider the usual Lebesgue spaces $L^r(\Om)$ and $L^r(\Om; \RN)$ equipped with the norm $\|\cdot \|_r$. When $1< r< \infty$ we denote by $W^{1, r}(\Om)$ and $W^{1,r}_0(\Om)$ the corresponding Sobolev spaces  equipped with the norms $\|\cdot \|_{1,r}$ and $\|\cdot \|_{1,r,0}$, respectively. By $r'$, we denote the conjugate of $r \in (1,\infty)$, that is, $\frac{1}{r}+\frac{1}{r'}=1$. 

For $i= 1, 2$ we define functions $\mathcal H_i\colon \Om \times [0, \infty) \to [0, \infty)$ by
\begin{align*}
        \mathcal H_i(x, t)= t^{p_i}+ \mu_i(x) t^{q_i},
\end{align*}
where $1< p_i< q_i< N$ and
\begin{align}\label{1}
        \frac{q_i}{p_i}< 1+ \frac{1}{N}, \qquad \mu_i\colon \close \to [0, \infty) \text{ is Lipschitz continuous.}
\end{align}
	
\begin{remark}
        From the condition above we easily see that  
        \begin{align*}
                q_i< p_i^*, \quad i= 1, 2,
        \end{align*}
		where $p_i^*$ is the critical Sobolev exponent of $p_i$ given by
		\begin{align*}
			p^*_i:= \frac{N p_i }{N- p_i}.
		\end{align*}
        Indeed, for fixed $ i \in \{1, 2\}$, we have to show that $ \displaystyle q_i< \frac{N p_i }{N- p_i}$, that is $Nq_i - p_i q_i< Np_i $. From condition \eqref{1} we have $ Nq_i - p_i< Np_i $. Moreover, since $q_i> 1$, we have 
        \begin{align*}
                N q_i - p_i q_i< N q_i - p_i< Np_i ,
        \end{align*}
        which shows the assertion.
\end{remark}

Let  $\displaystyle \rho_{\mathcal H_i}(u):= \into \mathcal H_i(x, |u|) dx$. Then the Musielak-Orlicz space $L^{\mathcal H_i}(\Om)$ is given by
\begin{align*}
        L^{\mathcal{H}_i}(\Omega)=\left \{u ~ \Big | ~ u: \Omega \to \R \text{ is measurable and } \rho_{\mathcal{H}_i}(u)< +\infty \right \}
\end{align*}
equipped with the Luxemburg norm
\begin{align*}
        \|u\|_{\mathcal H_i}:= \inf \l\{ \tau>0 : \, \rho_{\mathcal H_i}\l(\frac{u}{\tau}\r) \le 1 \r\}.
\end{align*}
Then $L^{\mathcal H_i}(\Om)$  becomes uniformly convex, and so a reflexive Banach space. Moreover we define the space
\begin{align*}
    L^{q_i}_{\mu_i}(\Omega)=\left \{u ~ \Big | ~ u: \Omega \to \R \text{ is measurable and } \into \mu_i(x) | u|^{q_i} \,dx< +\infty \right \}
\end{align*}
endowed with the seminorm 
\begin{align*}
        \|u\|_{q_i, \mu_i}:= \l(\into \mu_i(x) |u|^{q_i} dx \r)^{\frac{1}{q_i}}.
\end{align*}
From Colasuonno-Squassina \cite[Proposition 2.15]{Colasuonno-Squassina-2016} we have the following continuous embeddings
\begin{align*}
        L^{q_i}(\Om) \hookrightarrow L^{\mathcal H_i}(\Om) \hookrightarrow L^{p_i}(\Om) \cap L_{\mu_i}^{q_i}(\Om).
\end{align*}
For $u \ne 0$ we have $\displaystyle \rho_{\mathcal H_i}\l(\frac{u}{\|u\|_{\mathcal H_i}}\r)= 1$, so it is easy to see that
\begin{align}	\label{2}
        \min \left\{\|u\|_{\mathcal H_i}^{p_i}, \|u\|_{\mathcal H_i}^{q_i}\right\} 
        \leq \|u\|_{p_i}^{p_i}+ \|u\|_{q_i, \mu_i}^{q_i} 
        \leq \max \left\{\|u\|_{\mathcal H_i}^{p_i}, \|u\|_{\mathcal H_i}^{q_i}\right\}
\end{align}
for every $u \in L^{\mathcal H_i}(\Om)$. Then we can introduce the corresponding Sobolev space $W^{1, \mathcal H_i}(\Om)$ defined by
\begin{align*}
        W^{1, \mathcal H_i}(\Om):= \left\{u \in L^{\mathcal H_i}(\Om) : \, |\nabla u| \in L^{\mathcal H_i}(\Om) \right\}
\end{align*}
with the norm
\begin{align*}
        \|u\|_{1, \mathcal H_i}:= \|\nabla u \|_{\mathcal H_i}+ \|u\|_{\Hi},
\end{align*}
where $\|\nabla u\|_{\Hi}= \|\, |\nabla u| \, \|_{\Hi}$. 

Moreover, we write $W^{1, \Hi}_0(\Om):= \overline{C^{\infty}_0(\Om)}^{\|\cdot \|_{1, \Hi}}$ being the completion of $C^\infty_0(\Omega)$ in $W^{1, \mathcal H_i}(\Om)$. Taking \eqref{1} into account  we can refer to Colasuonno-Squassina \cite[Proposition 2.18]{Colasuonno-Squassina-2016} in order to consider an equivalent norm on $W^{1, \Hi}_0(\Om)$ given by
	\begin{align*}
	\|u\|_{1, \Hi, 0}= \|\nabla u\|_{\Hi}.
	\end{align*}
Note that $W^{1, \Hi}(\Om)$ as well as $W^{1, \Hi}_0(\Om)$ are uniformly convex, and so reflexive Banach spaces. 

Since $1< p_i< N$, we know that the embedding
\begin{align}\label{embedding_sobolev}
            W^{1,\mathcal{H}_i}_0(\Omega) \hookrightarrow \Lp{r_i}
\end{align}
is compact whenever $r_i<p^*_i$, see Colasuonno-Squassina \cite[Proposition 2.15]{Colasuonno-Squassina-2016}.
%
%
%
%

From equation \eqref{2} we directly obtain
\begin{align}\label{5}
    \begin{split}
        \min \l\{\|u\|_{1, \mathcal H_i, 0}^{p_i}, \|u\|_{1, \mathcal H_i, 0}^{q_i}\r\} 
        & \leq \|\nabla u\|_{p_i}^{p_i}+ \|\nabla u\|_{q_i, \mu_i}^{q_i}\\ 
        & \leq \max \l\{\|u\|_{1, \mathcal H_i, 0}^{p_i}, \|u\|_{1, \mathcal H_i, 0}^{q_i}\r\}
    \end{split}
\end{align}
for every $u \in W^{1, \Hi}_0(\Om)$. 

For $1< r< \infty$ we consider now the eigenvalue problem for the $r$-Laplacian with homogeneous Dirichlet boundary condition given by
	\begin{align}
	\label{6}
	\begin{aligned}
	-\Delta_r u&= \lambda |u|^{r-2} u \quad && \text{in } \Om, \\
	u&= 0 && \text{on } \rand.
	\end{aligned}
	\end{align}
Let us denote by $\lambda_{1, r}$ the first eigenvalue of \eqref{6}. It is well known that $\lambda_{1, r}$ is positive, simple, and isolated, see L{\^e} \cite{Le-2006}. Moreover, we have the following variational characterization 
	\begin{align}
	\label{7}
	\lambda_{1, r}= \inf_{u \in W^{1,r}_0(\Om)} \l\{\into |\nabla u|^r dx: \, \into |u|^r dx= 1\r\}.
	\end{align}
	
We now recall some definitions that we will use in the sequel. 

\begin{definition}\label{SplusPM}
    Let $X$ be a reflexive Banach space, $X^*$ its dual space and denote by $\langle \cdot \,, \cdot\rangle$ its duality pairing. Let $A\colon X\to X^*$, then $A$ is called
    \begin{enumerate}[leftmargin=1cm]
	\item[(a)]
	    to satisfy the $(\Ss_+$)-property if $u_n \weak u$ in $X$ and $\limsup_{n\to \infty} \langle Au_n,u_n-u\rangle \leq 0$ imply $u_n\to u$ in $X$;
	 \item[(b)]
	    pseudomonotone if $u_n \weak u$ in $X$ and $\limsup_{n\to \infty} \langle Au_n,u_n-u\rangle \leq 0$ imply $Au_n \weak u$ and $\langle Au_n,u_n\rangle \to \langle Au,u\rangle$;
    \item[(c)]
        coercive if
            \begin{align}
            \label{coercivity}
            \lim_{\|u\|_X \to \infty} \frac{\langle Au, u \rangle}{\|u\|_X}= \infty.
            \end{align}
    \end{enumerate}
\end{definition}

Our existence result is based on the following surjectivity result for pseudomonotone operators, see, e.\,g., Carl-Le-Motreanu \cite[Theorem 2.99]{Carl-Le-Motreanu-2007} or Papageorgiou-Winkert \cite[Theorem 6.1.57]{Papageorgiou-Winkert-2018}.

\begin{theorem}\label{theorem_pseudomonotone}
    Let $X$ be a real, reflexive Banach space, let $A\colon X\to X^*$ be a pseudomonotone, bounded, and coercive operator, and $b\in X^*$. Then, a solution of the equation $Au=b$ exists.
\end{theorem}

We consider the space $\mathcal{W}:=W^{1, \mathcal H_1}_0(\Om) \times W^{1, \mathcal H_2}_0(\Om)$ endowed with the norm 
	\begin{align*}
	\|(u, v)\|_{\mathcal{W}}:= \|u\|_{1, \mathcal H_1, 0}+ \| v \|_{1, \mathcal H_2, 0}, 
	\end{align*}
for every $(u, v) \in W^{1, \mathcal H_1}_0(\Om) \times W^{1, \mathcal H_2}_0(\Om)$.

Then we consider the operator $A\colon W^{1, \mathcal H_1}_0(\Om) \times W^{1, \mathcal H_2}_0(\Om) \to (W^{1, \mathcal H_1}_0(\Om))^* \times (W^{1, \mathcal H_2}_0(\Om))^* $ defined by
        \begin{align}\label{A}
            \begin{split}
                    \langle A(u, v), (\varphi, \psi)\rangle_{\mathcal H_1 \times \mathcal H_2}&:= \into \l(|\nabla u|^{p_1-2} \nabla u+ \mu_1(x) |\nabla u|^{q_1-2} \nabla u \r) \cdot \nabla \varphi dx \\
                    & \quad +  \into \l(|\nabla v|^{p_2-2} \nabla v+ \mu_2(x) |\nabla v|^{q_2-2} \nabla v \r) \cdot \nabla \psi dx,
            \end{split}
        \end{align}
        where $ \langle \cdot \, , \cdot \rangle_{\mathcal H_1 \times \mathcal H_2}$ is the duality pairing between $W^{1, \mathcal H_1}_0(\Om) \times W^{1, \mathcal H_2}_0(\Om)$ and its dual space $ (W^{1, \mathcal H_1}_0(\Om))^* \times (W^{1, \mathcal H_2}_0(\Om))^* $. The next result summarizes the  properties of the operator $A$.
        
\begin{lemma}\label{lemma1}
        Let $A\colon W^{1, \mathcal H_1}_0(\Om) \times W^{1, \mathcal H_2}_0(\Om) \to (W^{1, \mathcal H_1}_0(\Om))^* \times (W^{1, \mathcal H_2}_0(\Om))^* $ be the operator defined by \eqref{A}. Then, $A$ is bounded, continuous, monotone (hence maximal monotone), and of type $(S_+)$.
\end{lemma}

\begin{proof}
    The proof is similar to the one in Liu-Dai \cite[Proposition 3.1]{Liu-Dai-2018}. 
\end{proof}

\section{Main result}\label{section_3}
We assume the following hypotheses on the nonlinearities $f_1, f_2$. 

\begin{enumerate}
    \item[(H)] 
            $f_1, f_2 \colon \Omega \times \R \times \R \times \RN \times \RN \to \R  $ are Carath\'eodory functions such that
            \begin{enumerate}
                    \item[(i)] 
                            There exist $\alpha_i \in L^{\frac{r_i}{r_i-1}}(\Om) $ ($i=1,2$) such that
                            \begin{align*}
                                \begin{split}
                                    |f_1(x, s, t, \xi, \zeta)| 
                                    &\le A_1 |s|^{a_1}+ A_2 |t|^{a_2}+ A_3 |s|^{a_3} |t|^{a_4}+ A_4 |\xi|^{a_5} \\
                                    & \quad + A_5 |\zeta|^{a_6}+ A_6 |\xi|^{a_7} |\zeta|^{a_8}+  |\alpha_1(x)|, \\
                                    |f_2(x, s, t, \xi, \zeta)| 
                                    &\le B_1 |s|^{b_1}+ B_2 |t|^{b_2}+ B_3 |s|^{b_3} |t|^{b_4}+ B_4 |\xi|^{b_5} \\
                                    & \quad + B_5 |\zeta|^{b_6}+ B_6 |\xi|^{b_7} |\zeta|^{b_8}+  |\alpha_2(x)|,
                                \end{split}
                            \end{align*}
                            for a.\,a.\,$x \in \Om$, for all $s, t \in \R$ and for all $ \xi, \zeta \in \RN$, where $A_j, B_j$, $j=1, \dots 6$, are nonnegative constants and with $1< r_i< p_i^*$, $i= 1, 2$. Moreover,  the exponents $a_\ell, b_\ell $, $\ell= 1, \dots, 8$, are nonnegative and satisfy the following conditions
                            \begin{align*}
                                    \hspace*{1cm}
                                    \begin{array}[]{lll}
                                            \text{(E1)} \quad a_1 \le r_1- 1,
                                            & \text{(E2)} \quad \displaystyle a_2 \le \frac{r_1- 1}{r_1} r_2,\\[3ex]
                                            \text{(E3)} \quad  \displaystyle \frac{a_3}{r_1}+ \frac{a_4}{r_2} \le  \frac{r_1- 1}{r_1},  & \text{(E4)} \quad \displaystyle a_5 \le \frac{r_1- 1}{r_1} p_1,\\[3ex]
                                            \text{(E5)} \quad \displaystyle a_6 \le \frac{r_1- 1}{r_1} p_2,
                                            & \text{(E6)} \quad  \displaystyle \frac{a_7}{p_1}+ \frac{a_8}{p_2} \le  \frac{r_1- 1}{r_1},   \\[3ex]
                                            \text{(E7)} \quad b_1 \le \displaystyle\frac{r_2- 1}{r_2} r_1,
                                            & \text{(E8)} \quad \displaystyle b_2 \le r_2- 1,\\[3ex]
                                             \text{(E9)} \quad \displaystyle \frac{b_3}{r_1}+ \frac{b_4}{r_2} \le \frac{r_2- 1}{r_2},
                                            &\text{(E10)} \quad b_5 \le \displaystyle\frac{r_2- 1}{r_2} p_1,\\[3ex]
                                             \text{(E11)} \quad \displaystyle b_6 \le \frac{r_2- 1}{r_2} p_2,
                                            & \text{(E12)} \quad  \displaystyle \frac{b_7}{p_1}+ \frac{b_8}{p_2} \le  \frac{r_2- 1}{r_2}.  
                                    \end{array}
                            \end{align*}
                    \item[(ii)] 
                            There exist $\omega \in L^1(\Om)$ and $\Lambda, \Gamma \ge 0$ such that
                            \begin{align}\label{cond}
                                    \begin{split}
                                            & f_1(x, s, t, \xi, \zeta) s+f_2(x, s, t, \xi, \zeta) t\\ 
                                            & \le \Lambda \left(|\xi|^{p_1}+|\zeta|^{p_2}\right)+ \Gamma \left( |s|^{p_1}+|t|^{p_2}\right)+ \omega(x), 
                                    \end{split}
                            \end{align}
                            for a.\,a.\,$x \in \Om$, for all $s, t \in \R$ and for all $\xi, \zeta \in \R^N$ and with
                            \begin{align}\label{cond3}
                                    \Lambda +\Gamma \max\left\{\lambda_{1, p_1}^{-1},\lambda_{1, p_2}^{-1}\right\}< 1,
                            \end{align}
                            where $\lambda_{1, p_i}$ is the first eigenvalue of the $p_i$-Laplacian, see \eqref{6}.
            \end{enumerate}
\end{enumerate}

We say that $(u, v) \in W^{1, \mathcal H_1}_0(\Om) \times W^{1, \mathcal H_2}_0(\Om) $ is a weak solution of problem \eqref{syst} if
\begin{align}\label{weak}
        \begin{split}
                \into \l(|\nabla u|^{p_1-2} \nabla u+ \mu_1(x) |\nabla u|^{q_1-2} \nabla u \r) \cdot \nabla \varphi \,dx&= \into f_1(x, u, v, \nabla u, \nabla v) \varphi \,dx, \\
                \into \l(|\nabla v|^{p_2-2} \nabla v+ \mu_2(x) |\nabla v|^{q_2-2} \nabla v \r) \cdot \nabla \psi \,dx&= \into f_2(x, u, v, \nabla u, \nabla v) \psi \,dx,
        \end{split}
\end{align}
is satisfied for all test functions $(\varphi, \psi) \in W^{1, \mathcal H_1}_0(\Om) \times W^{1, \mathcal H_2}_0(\Om)$. Taking the embedding \eqref{embedding_sobolev} into account, along with the growth conditions in (H)(i), we see that the definition of a weak solution is well-defined. Indeed, if we estimate the integral concerning the function $f_1\colon \Omega \times \R \times \R \times \R^N \times \R^N \to \R$  by using condition (H)(i) we obtain several mixed terms. Let us consider, for example, the third term on the right-hand side of the growth of $f_1$. Applying H\"older's inequality we get
\begin{align*}
    \begin{split}
        &A_3\into |u|^{a_3}|v|^{a_4} \ph \,dx\\
        &\leq A_3\left(\into|u|^{a_3 s_1}\,dx\right)^{\frac{1}{s_1}}\left(\into|v|^{a_4 s_2}\,dx\right)^{\frac{1}{s_2}} \left(\into|\ph|^{ s_3}\,dx\right)^{\frac{1}{s_3}},
    \end{split}
\end{align*}
where $ (u,v) \in W^{1, \mathcal H_1}_0(\Om) \times W^{1, \mathcal H_2}_0(\Om) $, $\ph \in W^{1, \mathcal H_1}_0(\Om)$ and
\begin{align*}
        \frac{1}{s_1}+\frac{1}{s_2}+\frac{1}{s_3}=1.
\end{align*}
Taking $s_3=r_1$ with $1<r_1<p_1^*$ and using $s_1 \leq \frac{r_1}{a_3}$ as well as $s_2 \leq \frac{r_2}{a_4}$  leads to 
\begin{align*}
        \frac{a_3}{r_1}+ \frac{a_4}{r_2}\leq \frac{r_1-1}{r_1},
\end{align*}
which is exactly condition (E3). Note that the conditions in (H)(i) are chosen in order to prove our main results by applying the compact embedding \eqref{embedding_sobolev}. Of course, for the finiteness of the integrals in the weak formulation \eqref{weak}, we can also allow critical growth to have a well-defined weak formulation.

Now we are ready to formulate and prove our main result in this section.

\begin{theorem}\label{main}
        Let $1< p_i< q_i< N$, $i= 1, 2$, and let hypotheses \eqref{1} and (H) be satisfied. Then, there exists a weak solution $(u, v) \in W^{1, \mathcal H_1}_0(\Om) \times W^{1, \mathcal H_2}_0(\Om)$ of problem \eqref{syst}. 
\end{theorem}

\begin{proof}
    Let 
    \begin{align*}
        \hat N_{f_i}\colon W^{1, \mathcal H_1}_0(\Om) \times W^{1, \mathcal H_2}_0(\Om) \subset L^{r_1}(\Om) \times L^{r_2}(\Om) \to L^{r_1'}(\Om) \times L^{r_2'}(\Om) 
    \end{align*}
    be the Nemytskij operator associated to $f_i$. Moreover, let 
    \begin{align*}
        j_i^*\colon L^{r_1'}(\Om) \times L^{r_2'}(\Om) \to (W^{1, \mathcal H_1}_0(\Om))^* \times (W^{1, \mathcal H_2}_0(\Om))^* 
    \end{align*}
    be the adjoint operator for the embedding 
    \begin{align*}
        j_i\colon W^{1, \mathcal H_1}_0(\Om) \times W^{1, \mathcal H_2}_0(\Om) \to L^{r_1}(\Om) \times L^{r_2}(\Om).
    \end{align*}
    We then define 
    \begin{align*}
        N_{f_i}:= j_i^* \circ \hat N_{f_i}\colon W^{1, \mathcal H_1}_0(\Om) \times W^{1, \mathcal H_2}_0(\Om) \to (W^{1, \mathcal H_1}_0(\Om))^* \times (W^{1, \mathcal H_2}_0(\Om))^*,
    \end{align*}
    which is well-defined by hypotheses (H)(i). We set
    \begin{align}\label{mathcal-A}
        \mathcal A(u, v):= A(u, v)- N_{f_1}(u, v)- N_{f_2}(u, v).
    \end{align}
    Our aim is to apply Theorem \ref{theorem_pseudomonotone}. So, we need to show that $\mathcal A$ is bounded, pseudomonotone and coercive.
        
    {\bf Claim 1:} $\mathcal{A}$ is bounded.
        
    The boundedness of $\mathcal{A}$ follows directly from the boundedness of $A$ and the growth conditions on $f_1$ and $f_2$ stated in (H)(i).
        
    {\bf Claim 2:} $\mathcal{A}$ is pseudomonotone.

    To this end, let $\{(u_n, v_n)\}_{n \in \N} \subset W^{1, \mathcal H_1}_0(\Om) \times W^{1, \mathcal H_2}_0(\Om) $ be a sequence such that
    \begin{align}\label{weak lim}
        (u_n, v_n) \rightharpoonup (u, v) \quad \text{in } W^{1, \mathcal H_1}_0(\Om) \times W^{1, \mathcal H_2}_0(\Om)
    \end{align}
    and
	\begin{align}\label{limsup}
	    \limsup_{n \to \infty} \langle \mathcal A(u_n, v_n), (u_n- u, v_n- v) \rangle_{\mathcal H_1 \times \mathcal H_2} \le 0.
	\end{align}
    Taking the compact embedding \eqref{embedding_sobolev} into account yields
	\begin{align}\label{strong}
	    u_n \to u \quad \text{in } L^{r_1}(\Om) \quad \text{and} \quad v_n \to v \quad \text{in } L^{r_2}(\Om),
	\end{align}
	since $r_1< p_1^*$ and $r_2< p_2^*$, respectively. 
	We want to show that
	\begin{align}\label{lim-f}
	    \begin{split}
	        \lim_{n \to \infty} \into f_1(x, u_n, v_n, \nabla u_n, \nabla v_n)(u_n- u)\, dx&= 0, \\
	        \lim_{n \to \infty} \into f_2(x, u_n, v_n, \nabla u_n, \nabla v_n)(v_n- v) \,dx&= 0.
	    \end{split}
	\end{align}
    Let us consider the first expression in \eqref{lim-f}. By the growth condition (H)(i) it follows
	\begin{align}\label{f1}
	    \begin{split}
	        & \into f_1(x, u_n, v_n, \nabla u_n, \nabla v_n) (u_n-u) \,dx \\
	        & \leq \into \bigg(A_1 |u_n|^{a_1}+ A_2 |v_n|^{a_2}+ A_3 |u_n|^{a_3} |v_n|^{a_4} \\
	        & \quad+ A_4 |\nabla u_n|^{a_5}+ A_5 |\nabla v_n|^{a_6}+ A_6 |\nabla u_n|^{a_7} |\nabla v_n|^{a_8}+ |\alpha_1(x)|\bigg) |u_n- u| \,dx.
	    \end{split}
	\end{align}
    Applying H\"older's inequality, \eqref{strong} and conditions (E1) and (E2), respectively, we obtain
	\begin{align*}
	    \begin{split}
	        A_1 \into |u_n|^{a_1} |u_n- u| \,dx 
	        & \leq A_1 \l(\into |u_n|^{a_1 r_1'}\,dx \r)^{\frac{1}{r_1'}} \|u_n- u\|_{r_1} \\
	        & \le C_{1} \l(1+ \|u_n\|_{r_1}^{r_1-1}\r) \|u_n- u\|_{r_1} \to 0
	    \end{split}
	\end{align*}
    and
	\begin{align*}
	    \begin{split}
	        A_2 \into |v_n|^{a_2} |u_n- u| \,dx & \le A_2 \l(\into v_n^{a_2 r_1'} \,dx \r)^{\frac{1}{r_1'}} \| u_n- u\|_{r_1} \\
	        & \le C_{2} \l(1+ \|v_n\|_{r_2}^{\frac{r_2}{r_1'}} \r) \|u_n- u\|_{r_1} \to 0
	    \end{split}
	\end{align*}
	for some $C_1, C_2>0$.
    Moreover, H\"older's inequality with exponents $x_1, y_1, z_1> 1$ such that
	\begin{align*}
	    x_1 a_3 \le r_1, \quad y_1 a_4 \le r_2, \quad z_1= r_1, \quad \frac{1}{x_1}+ \frac{1}{y_1}+ \frac{1}{z_1}= 1
	\end{align*}
    gives, by hypothesis (E3),
	\begin{align*}
	    A_3 \into |u_n|^{a_3} |v_n|^{a_4} |u_n- u| \,dx \le A_3 \|u_n\|_{a_3 x_1}^{a_3} \|v_n \|_{a_4 y_1}^{a_4} \|u_n- u\|_{r_1} \to 0.
	\end{align*}
    Next we apply H\"older's inequality with exponents $r_1, r_1'$ and use (E4) and (E5) to get
	\begin{align*}
	    \begin{split}
	        A_4 \into |\nabla u_n|^{a_5} |u_n- u| \,dx 
	        & \leq A_4 \l(\into |\nabla u_n|^{a_5 r_1'} \,dx \r)^{\frac{1}{r_1'}} \|u_n- u\|_{r_1} \\
	        & \leq C_3 \l(1+ \|\nabla u_n\|_{p_1}^{\frac{p_1}{r_1'}} \r) \|u_n- u\|_{r_1} \to 0
	    \end{split}
	\end{align*}
    and
	\begin{align*}
	    \begin{split}
	        A_5 \into |\nabla v_n|^{a_6} |u_n- u| \,dx 
	        & \leq A_5 \l(\into |\nabla v_n|^{a_6 r_1'} \,dx \r)^{\frac{1}{r_1'}} \|u_n- u\|_{r_1} \\
	        & \leq C_4 \l(1+ \|\nabla v_n\|_{p_2}^{\frac{p_2}{r_1'}} \r) \|u_n- u\|_{r_1} \to 0
	\end{split}
	\end{align*}
for some $C_3, C_4>0$. Furthermore, condition (E6) allows us to apply H\"older's inequality with exponents $x_2, y_2, z_2> 1$ such that
	\begin{align*}
	    x_2 a_7 \leq p_1, \quad y_2 a_8 \leq p_2, \quad z_2= r_1, \quad \frac{1}{x_2}+ \frac{1}{y_2}+ \frac{1}{z_2}= 1
	\end{align*}
    in order to have
	\begin{align*}
	    A_6 \into |\nabla u_n|^{a_7} |\nabla v_n|^{a_8} (u_n- u) \,dx 
	    \leq A_6 \|\nabla u_n \|_{a_7 x_2}^{a_7} \|\nabla v_n\|_{a_8 y_2}^{a_8} \|u_n- u\|_{r_1} \to 0,
	\end{align*}
    since both $\|\nabla u_n \|_{a_7 x_2}$ and $\|\nabla v_n\|_{a_8 y_2}$ are bounded. Finally, for the last term in \eqref{f1} we have
	\begin{align*}
	    \into |\alpha_1(x)| (u_n- u) \,dx 
	    \leq \|\alpha_1\|_{r_1'} \|u_n- u\|_{r_1} \to 0.
	\end{align*}
    Combining all the calculations above give 
	\begin{align*}
	    \lim_{n \to \infty} \into f_1\l(x, u_n, v_n, \nabla u_n, \nabla v_n\r)(u_n- u)\,dx= 0.
	\end{align*}
    Applying similar arguments proves that
	\begin{align*}
	    \lim_{n \to \infty} \into f_2\l(x, u_n, v_n, \nabla u_n, \nabla v_n\r) (v_n- v) \,dx= 0.
	\end{align*}
	Hence, \eqref{lim-f} is fulfilled. We now take the weak formulation \eqref{weak}, replace $u$ by $u_n$, $v$ by $v_n$, $\varphi$ by $u_n- u$ and $\psi $ by $v_n- v$ and use \eqref{limsup} as well as \eqref{lim-f}  in order to have
	\begin{align}\label{limsup1}
	    \begin{split}
	        & \limsup_{n \to \infty} \langle A(u_n, v_n), (u_n-u, v_n- v)\rangle_{\mathcal H_1 \times \mathcal H_2}\\
	        & = \limsup_{n \to \infty} \langle \mathcal A(u_n, v_n), (u_n- u, v_n- v)\rangle_{\mathcal H_1 \times \mathcal H_2} \le 0.
        \end{split}
	\end{align}
    Since $A$ satisfies the $(S_+)$-property, see Lemma \ref{lemma1}, we derive from \eqref{weak lim} and \eqref{limsup1} that
	\begin{align*}
	    (u_n, v_n) \to (u, v) \quad \text{in } W^{1, \mathcal H_1}_0(\Om) \times W^{1, \mathcal H_2}_0(\Om).
	\end{align*}
    Since $\mathcal A$ is continuous we have $\mathcal A(u_n, v_n) \to \mathcal A(u, v)$ in $(W^{1, \mathcal H_1}_0(\Om))^* \times (W^{1, \mathcal H_2}_0(\Om))^*$, which proves that $\mathcal A $ is pseudomonotone.  

    {\bf Claim 3:} $\mathcal{A}$ is coercive.
    
    First of all, taking into account the representation \eqref{7} and replacing $r$ by $p_1$ and $p_2$, respectively, we have
	\begin{align}\label{eigenv}
	    \|u\|_{p_1}^{p_1} \leq \lambda_{1, p_1}^{-1} \|\nabla u \|_{p_1}^{p_1} \quad \text{and} \quad \|v\|_{p_2}^{p_2} \leq \lambda_{1, p_2}^{-1} \|\nabla v \|_{p_2}^{p_2},
	\end{align}
    for all $(u, v) \in W^{1, \mathcal H_1}_0(\Om) \times W^{1, \mathcal H_2}_0(\Om)$. Note that $W^{1, \mathcal H_1}_0(\Om) \subseteq \Wpzero{p_1}$ and $W^{1, \mathcal H_2}_0(\Om)\subseteq \Wpzero{p_2}$. Applying these facts along with \eqref{eigenv}, \eqref{cond}, and \eqref{5} leads to
	\begin{align*}
	    \begin{split}
	        &\langle \mathcal A(u, v), (u, v)\rangle_{\mathcal H_1 \times \mathcal H_2}\\
	        &= \into \l(|\nabla u|^{p_1-2} \nabla u+ \mu_1(x) |\nabla u|^{q_1-2} \nabla u\r) \cdot \nabla u \,dx \\
	        & \quad+ \into \l(|\nabla v|^{p_2-2} \nabla v+ \mu_2(x) |\nabla v|^{q_2-2} \nabla v\r) \cdot \nabla v \,dx \\
	        & \quad- \into f_1(x, u, v, \nabla u, \nabla v)u \,dx- \into f_2(x, u, v, \nabla u, \nabla v)v \,dx \\
	        & \geq \|\nabla u\|_{p_1}^{p_1}+ \|\nabla u\|_{q_1, \mu_1}^{q_1}+ \|\nabla v \|_{p_2}^{p_2}+ \|\nabla v\|_{q_2, \mu_2}^{q_2} \\
	        & \quad-\Lambda \left(\|\nabla u\|_{p_1}^{p_1}+\|\nabla v\|_{p_2}^{p_2}\right)- \Gamma \left( \|u\|_{p_1}^{p_1}+\|v\|_{p_2}^{p_2}\right)- \|\omega\|_1\\
        	& \geq \l(1- \Lambda- \Gamma \lambda_{1, p_1}^{-1}\r) \|\nabla u\|_{p_1}^{p_1}
        	+ \|\nabla u\|_{q_1, \mu_1}^{q_1} \\
        	& \quad +\l(1- \Lambda- \Gamma \lambda_{1, p_2}^{-1}\r) \|\nabla v\|_{p_2}^{p_2}+ \|\nabla v \|_{q_2, \mu_2}^{q_2}- \|\omega\|_1 \\
        	& \geq \bigg( 1- \Lambda- \Gamma \max\left\{\lambda_{1, p_1}^{-1},\lambda_{1, p_2}^{-1}\right\}\bigg)\left(\min\left\{\|u \|_{1, \mathcal H_1, 0}^{p_1}, \|u\|_{1, \mathcal H_1, 0}^{q_1}\right\} \right.\\
        	& \quad \left.+ \min\left\{\|v \|_{1, \mathcal H_2, 0}^{p_2}, \|v\|_{1, \mathcal H_2, 0}^{q_2}\right\}\right) -\|\omega\|_{L^1(\Om)}.
	    \end{split}
	\end{align*}	
    Since $1< p_i< q_i$ and  condition \eqref{cond3} holds, it follows that  \eqref{coercivity} is satisfied, and hence $\mathcal A$ is coercive. 

    From the Claims 1--3 we see that $\mathcal A$ is bounded, pseudomonotone and coercive. Therefore, by Theorem \ref{theorem_pseudomonotone}, there exists $(u, v) \in W_0^{1, \mathcal H_1}(\Om) \times W_0^{1, \mathcal H_2}(\Om)$ such that $\mathcal A(u, v)= 0$. Taking into account the definition of $\mathcal A$, see equation \eqref{mathcal-A}, it follows that $(u, v)$ is a weak solution of problem \eqref{syst}. That finishes the proof.
\end{proof}

\section{A uniqueness result}\label{section_4}

Now we consider the uniqueness of solutions of \eqref{syst}. To this end, let ${\bf f}\colon \Omega \times \R^2 \times (\R^N)^2 \to \R^2$ be the vector field defined by
\begin{align*}
    {\bf f}(x,s,\xi)=(f_1(x,s,\xi),f_2(x,s,\xi))
\end{align*}
for a.\,a.\,$x\in\Omega$, for all $s \in \R^2$ and for all $\xi \in (\R^N)^2$. We suppose the following conditions on $f$:
\begin{enumerate}
    \item[(U1)]
        There exists $c_1 \geq 0$ such that
        \begin{align*}
            ({\bf f}(x,s,\xi)-{\bf f}(x,t,\xi))\cdot (s-t) \leq c_1 |s-t|^2
        \end{align*}
        for a.\,a.\,$x\in\Omega$, for all $s,t \in \R^2$ and for all $\xi \in (\R^N)^2$.
    \item[(U2)]
        There exist $\rho=(\rho_1,\rho_2)$ with $\rho_i \in \Lp{s_i}$, $1<s_i<p_i^*$ and $c_2 \geq 0$ such that
        ${\bf f}(x,s,\cdot)-\rho(x)$ is linear on $(\R^N)^2$ for a.\,a.\,$x\in\Omega$ and for all $s\in \R^2$ and
        \begin{align*}
            |{\bf f}(x,s,\xi)-\rho(x)| \leq c_2 |\xi|
        \end{align*}
        for a.\,a.\,$x\in\Omega$, for all $s \in \R^2$ and for all $\xi \in (\R^N)^2$.
\end{enumerate}

Our main result in this section reads as follows.

\begin{theorem}\label{theorem_uniqueness}
    Let \eqref{1}, (H), (U1), and (U2) be satisfied. If $2= p_i< q_i< N$ for $i=1,2$ and
    \begin{align}\label{condition_uniqueness}
        c_1 \lambda_{1,2}^{-1}+c_2 \left(2\lambda_{1,2}^{-1}\right)^{\frac{1}{2}}<1,
    \end{align}
    then there exists a unique weak solution of problem \eqref{syst}. 
\end{theorem}

\begin{proof}
    Let $u=(u_1, u_2), v=(v_1, v_2) \in W_0^{1, \mathcal H_1}(\Om) \times W_0^{1, \mathcal H_2}(\Om)$ be two weak solutions of \eqref{syst}. Considering the weak formulation for $u$ and $v$, choosing $\varphi= u_1-v_1$ as well as $\psi= u_2- v_2$ and subtracting the related equations gives
	\begin{align}\label{u1}
	    \begin{split}
	        & \into |\nabla (u_1- v_1)|^2\,dx 
	        + \into |\nabla(u_2- v_2)|^2\,dx\\
	        & \quad + \into \mu_1(x) \l(|\nabla u_1|^{q_1-2} \nabla u_1- |\nabla v_1|^{q_1- 2} \nabla v_1 \r) \cdot \nabla(u_1- v_1)\,dx \\
	        & \quad + \into \mu_2(x) \l(|\nabla u_2|^{q_2-2} \nabla u_2- |\nabla v_2|^{q_2-2} \nabla v_2 \r) \cdot \nabla (u_2- v_2) \,dx \\
	        &= \into \l({\bf f}(x, u, \nabla u)- {\bf f}(x, v, \nabla u)\r) \cdot (u- v) \,dx \\
	        & \quad+ \into \l({\bf f}(x, v, \nabla u)-\rho(x)- {\bf f}(x, v, \nabla v)+\rho(x)\r)\cdot (u-v)\, dx.
	    \end{split}
	\end{align}
	By the monotonicity of $\xi \mapsto |\xi|^{q_i-2}\xi$ we see that the third and the fourth integral on the left-hand side of \eqref{u1} are nonnegative, that is,
	\begin{align}\label{u2}
	    \begin{split}
	        & \into |\nabla (u_1- v_1)|^2\,dx 
	        + \into |\nabla(u_2- v_2)|^2\,dx\\
	        & \quad + \into \mu_1(x) \l(|\nabla u_1|^{q_1-2} \nabla u_1- |\nabla v_1|^{q_1- 2} \nabla v_1 \r) \cdot \nabla(u_1- v_1)\,dx \\
	        & \quad + \into \mu_2(x) \l(|\nabla u_2|^{q_2-2} \nabla u_2- |\nabla v_2|^{q_2-2} \nabla v_2 \r) \cdot \nabla (u_2- v_2) \,dx \\
	        & \geq \into |\nabla (u_1- v_1)|^2\,dx 
	        + \into |\nabla(u_2- v_2)|^2\,dx\\
	        & = \|\nabla (u_1-v_1)\|_2^2+\|\nabla (u_2-v_2)\|_2^2.
	    \end{split}
	\end{align}
	On the other side, by applying (U1) to the first integral on the right-hand side of \eqref{u1} and (U2) to the second we obtain along with H\"older's inequality 
	\begin{align}\label{u3}
	    \begin{split}
	        &\into \l({\bf f}(x, u, \nabla u)- {\bf f}(x, v, \nabla u)\r) \cdot (u- v) \,dx \\
	        & \quad+ \into \l({\bf f}(x, v, \nabla u)-\rho(x)- {\bf f}(x, v, \nabla v)+\rho(x)\r)\cdot (u-v)\, dx\\
	        & \leq c_1 \left(\|u_1-v_1\|_2^2+\|u_2-v_2\|_2^2 \right) \\
	        & \quad+ \into (f_1(x,v_1,v_2,(u_1-v_1)\nabla (u_1-v_1),(u_1-v_1)\nabla (u_2-v_2))-\rho_1(x) )\,dx\\ 
	        & \quad + \into (f_2(x,v_1,v_2,(u_2-v_2)\nabla (u_1-v_1),(u_2-v_2)\nabla (u_2-v_2))-\rho_2(x) )\,dx\\
	        & \leq c_1 \lambda_{1,2}^{-1}\left(\|\nabla (u_1-v_1)\|_2^2+\|\nabla (u_2-v_2)\|_2^2 \right)\\
	        & \quad +c_2 \into \left(|u_1-v_1|+|u_2-v_2| \right) \left(|\nabla (u_1-v_2)|^2+|\nabla (u_2-v_2)|^2\right)^{\frac{1}{2}}\,dx\\
	        & \leq \left(c_1 \lambda_{1,2}^{-1}+c_2 \left(2\lambda_{1,2}^{-1}\right)^{\frac{1}{2}}\right)\left(\|\nabla (u_1-v_1)\|_2^2+\|\nabla (u_2-v_2)\|_2^2 \right).
	    \end{split}
	\end{align}
	Combining \eqref{u1}, \eqref{u2} and \eqref{u3} gives
	\begin{align}\label{u4}
	    \begin{split}
	        & \|\nabla (u_1-v_1)\|_2^2+\|\nabla (u_2-v_2)\|_2^2\\
	        & \leq \left(c_1 \lambda_{1,2}^{-1}+c_2 \left(2\lambda_{1,2}^{-1}\right)^{\frac{1}{2}}\right)\left(\|\nabla (u_1-v_1)\|_2^2+\|\nabla (u_2-v_2)\|_2^2 \right).
	    \end{split}
	\end{align}
	Taking \eqref{condition_uniqueness} into account, we see from \eqref{u4} that $u_1=v_1$ and $u_2=v_2$ and so the solution of \eqref{syst} is unique.
\end{proof}

\end{document}